\documentclass[12pt,a4paper]{amsart}
\usepackage{amsmath,amsthm,amssymb,amsfonts}

\usepackage{graphicx}
\usepackage{color}

\providecommand{\U}[1]{\protect\rule{.1in}{.1in}}

\newtheorem{theorem}{Theorem}[section]
\newtheorem{proposition}[theorem]{Proposition}
\newtheorem{corollary}[theorem]{Corollary}

\newtheorem{remarks}[theorem]{Remarks}

\newtheorem{questions}[theorem]{Questions}

\numberwithin{equation}{section}

\title[Reflexivity and non-weakly null maximizing sequences]{Reflexivity and non-weakly null maximizing sequences}

\author[R. M. Aron]{Richard M. Aron}
    \address{ Department of Mathematical
Sciences\\ Kent State University\\ Kent,  Ohio 44242,  USA}
 \email{aron@math.kent.edu}

 \author[D. Garc\'ia]{Domingo Garc\'ia}
\address{Departamento de An\'{a}lisis Matem\'{a}tico,
Universidad de Valencia, Doctor Moliner 50, 46100 Burjasot (Valencia), Spain}
\email{domingo.garcia@uv.es}

\author[D. Pellegrino]{Daniel Pellegrino}
\address{Departamento de Matem\'{a}tica, UFPB, Jo\~{a}o Pessoa, PB, Brazil}
\email{dmpellegrino@gmail.com}

\author[E. V. Teixeira]{Eduardo V. Teixeira}
	\address{
	Department of Mathematics,
	\indent University of Central Florida,
	\indent Orlando, FL, USA 32816}
\email{Eduardo.Teixeira@ucf.edu}

\thanks{The  first and
second authors were supported by  MINECO and FEDER project  MTM2017-83262-C2-1-P. The second  author was also supported by  PROMETEO/2017/102 of the Generalitat Valenciana. The third author was supported by CNPq-Grant 307327/2017-5. The fourth author was supported by FUNCAP/CNPq/PRONEX Grant 00068.01.00/15}



\begin{document}

\begin{abstract}
We introduce and explore a new property related to reflexivity that plays an
important role in the characterization of norm attaining operators. We also present an
application to the theory of compact perturbations of linear operators and characterize
norm attaining scalar-valued continuous $2$-homogeneous polynomials on $\ell_{2}$.
\end{abstract}

\maketitle

\section{Introduction}

The concept of reflexivity plays an important role in functional analysis and is closely
related to the theory of norm attaining operators. A classical result due to R. C. James
\cite{j2, James} asserts that a (real or complex) Banach space $E$ is reflexive if and
only if all bounded linear functionals on $E$ attain their norms, i.e., the norm of a bounded
functional, which is a supremum, is in fact a maximum. Since then, the relation between
reflexivity and norm attainment has attracted the interest of many mathematicians
(see, e.g., \cite{ak, CheLou, G-P, J-M, KimLee}).

Let ${E}$ and ${F}$ be Banach spaces over $\mathbb{K}=\mathbb{R}$ or $\mathbb{C}$ and
$\mathcal{L}(E,F)$ be the space of all bounded linear operators from $E$ into $F$. The
topological dual of $E$ will be represented by $E^{\ast}$. A bounded linear operator
${T}\colon E\rightarrow F$ is said to {\em{attain its norm}}
if there exists a
norm one vector $x\in E$ so that
\[
\Vert{T(x)}\Vert=\Vert{T}\Vert.
\]

A {\em{maximizing sequence}} for $T\in\mathcal{L}(E,F)$ is a sequence $(x^{n})_{n=1}^\infty$ in $E$ with
$\Vert x^{n}\Vert=1$ for all $n$, so that $\Vert T(x^{n})\Vert$ converges to $\Vert
T\Vert$ as $n\rightarrow \infty$. In this note we introduce and explore the following
property:

\bigskip

\textbf{Weak Maximizing Property. } A pair of Banach spaces $\left( E,F\right)  $ is said
to have the {\it weak maximizing property} if for any bounded linear operator
$T:E\rightarrow F$, the existence of a non-weakly  null maximizing sequence for $T$
implies that $T$ attains its norm. When $E=F$ we say that $E$ has the {\it weak maximizing
property}.

\bigskip

In some sense this new property extends the notion of reflexivity, since  the pair $\left(  E,F\right)$ has the weak maximizing property for
some $F\neq\left\{ 0\right\}$ if and only if $E$ is reflexive (see Corollary \ref{uuio}).
For linear operators, the weak maximizing property should be explored for infinite dimensional
target spaces $F$, since every maximizing sequence for a non-zero, finite rank bounded linear operator is non-weakly null. On the other hand, for homogeneous polynomials
(see Section 3) the scalar case deserves attention since the analogue of the weak
maximizing property does not characterize reflexive spaces.

We will show that, like reflexivity, this new concept is strongly connected to the theory of norm attaining operators.

The following  result can be regarded as the first connection between the weak maximizing property and norm attaining operators.

\begin{theorem}(\cite[Theorem 1]{PT})
\label{ptp} Let $1<p<\infty$, $1\leq q<\infty$ and $T\colon\ell_{p}%
\rightarrow\ell_{q}$ be a bounded linear operator. Then $T$ attains its norm
if, and only if, there exists a non-weakly null maximizing sequence for $T$.
In particular, $\left(  \ell_{p},\ell_{q}\right)  $ has the weak maximizing
property$.$
\end{theorem}

This paper is organized as follows. We believe that it is important to understand
the present limitations of our knowledge about the weak maximizing property. For this,
we begin Section 2 with a brief outline of the proof of Theorem \ref{ptp}. We then observe that a full
version of Theorem \ref{ptp} is true for more general $\ell_{p}$ spaces. This allows us
to deduce that $\left( \ell_{p}(\Gamma_{1}),\ell_{q}(\Gamma_{2})\right) $ has the weak maximizing property
whenever $1<p<\infty$, $1\leq q<\infty$, where
$\Gamma_{1}$ and $\Gamma_{2}$ are arbitrary index sets. In addition, an application of
this result is given in the context of compact perturbations of continuous linear operators.
Albeit simple, our application (see Proposition \ref{sepG}) reinforces the special role
played by the weak maximizing property and substantially improves
a result of J. Kover
\cite{quae} on compact perturbations of operators acting on Hilbert spaces. As a
consequence of James' characterization of reflexivity and the results
in this section, we
conclude that if $E$ has the weak maximizing property, then $E$ is reflexive. So, it seems natural to ask whether the reflexivity of $E$ characterizes the weak maximizing property for $E$.

\bigskip

Any answer to this problem will lead us in interesting directions. A positive
answer will show a new feature of reflexivity in the theory of norm attaining
operators. On the other hand, a negative answer will highlight the importance
of studying the weak maximizing property.\\

In the final section we investigate possible extensions of Theorem \ref{ptp} to the
setting of homogeneous polynomials. Recall that a continuous $m$-homogeneous
polynomial $P\colon E\longrightarrow F$ is a mapping so that%
\[
P(x)=T_{P}(x,...,x)
\]
for some (unique) symmetric continuous $m$-linear mapping
$T_{P}:{E\times \cdots\times   E}\rightarrow F$.
The space of all continuous $m$-homogeneous polynomials (with the $\sup$
norm) $P:E\rightarrow F$ is denoted by $\mathcal{P}(^{m}E;F)$. The concept of maximizing
sequence is naturally extended to polynomials and an unavoidable question is whether
an analogue of
Theorem \ref{ptp} is valid for polynomials. It is worth recalling that contrary to the
case of linear operators, compact polynomials on reflexive Banach spaces are not
necessarily norm attaining.

A pair of Banach spaces $\left( E,F\right)  $ is said
to have the {\it $m-$homogeneous polynomial weak maximizing property} $(m\in \mathbb{N})$ if for any continuous
$m$-homogeneous polynomial $P:E\rightarrow F$, the existence of a non-weakly null maximizing sequence for $P$
implies that $P$ attains its norm.

We will show by means of a counterexample that there is no analogue of Theorem \ref{ptp}
for polynomials of degree $m\geq3$. Surprisingly, for $m=2,$ $E=\ell_{2}$ and
$F=\mathbb{K}$ a satisfactory extension of Theorem \ref{ptp} is valid. In other words, we
show in Theorem 3.1 that a $2-$homogeneous polynomial $P:\ell_2 \to \mathbb{K}$ attains its norm if and
only if there is a non-weakly null sequence $(x^n)_{n=1}^\infty \subset \ell_2$ of unit vectors such that
$|P(x^n)| \to \|P\|$, i.e., the pair $\left(  \ell_{2},\mathbb{K} \right)  $ has the
$2-$homogeneous polynomial weak maximizing property.

\section{Weak maximizing property for pairs of $\ell_{p}(\Gamma)$ and an application}

We begin this section with the promised sketch of the proof of the non-trivial direction of Theorem \ref{ptp}:

\begin{proof}[Sketch of the proof of  Theorem \ref{ptp}.]
Let $(u^n)_{n=1}^\infty \subset
\ell_p$ be a non-weakly null maximizing sequence for $T.$ By reflexivity, we may assume that $(u^n)_{n=1}^\infty$ converges weakly to
$u \neq 0.$ One can choose a further subsequence, still calling it $(u^n)_{n=1}^\infty,$ with the property that the sequence
$\big ( |T(u^n)(e_i) - T(u)(e_i)|^{q-1} \big )_{i=1}^\infty   $ converges weakly to $0$ in $\ell_{\frac{q}{q-1}}.$\\

Next, by calculus one shows that for fixed $r > 1$ and $\varepsilon > 0,$ there is a positive constant $C_\varepsilon$ such that
the following inequality holds for every $x \in \mathbb{R}$

\begin{equation}\label{real}
\big | |x|^r - |x-1|^r - 1 \big | \leq C_\varepsilon |x-1|^{r-1} + \varepsilon^r + \sup_{|t-1| \leq \varepsilon} |\ |t|^r -1 |.
\end{equation}

For fixed $n,$ by applying estimate $\eqref{real}$ to $x_i = \frac{T(u^n)(e_i)}{T(u)(e_i)}$ whenever the denominator is non-zero and adding over $i,$ we get

\begin{equation}\label{coordinates}
\begin{split}
\sum_i {\big |} |T(u^n)(e_i)|^q - |T(u^n)(e_i) - T(u)(e_i)|^q - |T(u)(e_i)|^q {\big | }
\leq \\
C_\varepsilon \Delta_n + \delta(\varepsilon) \|T(u)\|_{\ell_q}^q,
\end{split}
\end{equation}
where $\Delta_n = \sum_i |T(u)(e_i)| |T(u^n)(e_i) - T(u)(e_i)|^{q-1}.$

Since the sequence $(|T(u^n)(e_i) - T(u)(e_i)|^{q-1})_{i=1}^\infty$ is
weakly null in $\ell_{\frac{q}{q-1}}$ and $T(u) \in \ell_q = (\ell_{\frac{q}{q-1}})^\ast,$ it follows that $\Delta_n \to 0$ as $n \to \infty.$
Letting $n \to \infty$ in $\eqref{coordinates}$ and then letting $\varepsilon \to 0,$ we conclude that

\begin{equation}\label{equality_T}
\|T(u^n)\|^q_{\ell_{q}} =
 \|T(u)\|^q_{\ell_q} + \|T(u^n - u)\|^q_{\ell_q} + \mathrm{o}(1).
\end{equation}

In a similar way, one shows that for a further subsequence $(u^n)_{n=1}^\infty$ of the original sequence, one has

\begin{equation}\label{equality_u}
\|u^n - u\|^p_{\ell_p} =1 - \|u\|^p_{\ell_p} + \mathrm{o}(1).
\end{equation}

Combining $\eqref{equality_T}$ and $\eqref{equality_u},$ we see that

\begin{equation}\label{final}
\|T(u^n)\|^p_{\ell_q} \leq \|T(u)\|^p_{\ell_q} + \|T\|^p \big(1 - \|u\|^p_{\ell_p} + \mathrm{o}(1)\big) + \mathrm{o}(1).
\end{equation}

Let $n \to \infty,$
it easily follows that $\|T(u)\|_{\ell_p} \geq \|T\| \|u\|_{\ell_p}.$
\end{proof}

\vspace{.5cm}

In connection with Theorem \ref{ptp} and its proof, the following comments are pertinent.
\begin{remarks}
\noindent (a) Observe that the proof makes great use of the particular norms of both the domain and the
range of the operator. In particular, the argument does not hold for renormings of either $\ell_p$ or $\ell_q$
and, indeed, the authors do not know if the result is even true for other norms.\\
\noindent (b) Note that the norm of the operator $T$ is attained at the normalization $\frac{u}{||u||}$ of the
non-zero limit of the maximizing sequence. This should be compared with several examples in Section~3.\\
\end{remarks}

An important question that remains open is whether the pair $(L_p, L_q)$ has the
weak maximizing property, for some range of exponents $(p,q)$. In the case
$p\le q$ the same argument used in the proof of Theorem 1.1 (and Proposition \ref{second-result}
below) can be adapted if one substitutes weak convergence by almost
everywhere convergence of the maximizing sequence and its image.

However, for $\ell_{p}(\Gamma)$ spaces, we have a similar version to Theorem \ref{ptp}.
Note that in the proof below we also prove a version of Pitt's Theorem (see \cite{Pitt})
for $\ell_{p}(\Gamma)$ spaces.

\begin{proposition}\label{second-result}
If $1<p<\infty$, $1\leq q<\infty$ and $\Gamma_{1}$ and $\Gamma_{2}$ are
arbitrary index sets, then the pair $\left( \ell_{p}(\Gamma_{1}),\ell_{q}(\Gamma_{2})\right) $
has the weak maximizing property.

\end{proposition}

\begin{proof}
Let $T\colon\ell_{p}(\Gamma_{1}) \rightarrow\ell_{q}(\Gamma_{2})$ be a bounded linear
operator. The case $1<p\leq q<\infty$ can be obtained, \textit{mutatis mutandis}, from the proof of Theorem \ref{ptp}.

A standard argument shows that when $p>q,$ any bounded linear operator attains its norm.  To see this,
suppose that $T\colon\ell_{p}(\Gamma_{1})\rightarrow\ell_{q}(\Gamma_{2})$ is a bounded linear
operator. Repeating the standard argument, we claim that $T$ is compact. If $T$ were not compact, there would be a sequence
of unit vectors $(x^{n})_{n=1}^\infty$ in $\ell_{p}(\Gamma_{1})$ such
that the sequence $(T(x^{n}))_{n=1}^\infty$ has no limit point. Let $\Lambda_{n}%
=\{i\in\Gamma_{1} : \ x^{n,i}\neq0\}$ and $\Lambda^{\prime}=\cup\Lambda_{n}.$ Now, consider
the restriction of $T$ to $\ell_{p}(\Lambda^{\prime})$, which is essentially $\ell_{p}$
(since $\Lambda^{\prime}$ is countable). Then we would have a non-compact bounded linear
operator $T\colon\ell_{p}\rightarrow\ell _{q}$. This contradicts the classical Pitt theorem
and the proof is complete.

\end{proof}

Although the previous proof covers the case of Hilbert spaces, we think that the
following sketch of a direct proof when $T:H_1 \to H_2$, for $H_1$ and $H_2$  Hilbert spaces,  deserves to be mentioned. For that, let $(u^{n})_{n=1}^\infty$ be a
non-weakly  null maximizing sequence for $T$. Passing to a subsequence if necessary, we
may assume that $(u^{n})_{n=1}^\infty$ converges weakly to some $u\not =0.$ A simple and direct
handling of inner-product properties gives us
$$1 - \left\Vert u^{n}-u\right\Vert ^{2}=\left\Vert u\right\Vert ^{2}%
+\mathrm{o}(1)$$ and $$\left\Vert T(u^{n})\right\Vert ^{2} - \left\Vert
T(u^{n})-T(u)\right\Vert ^{2} = \left\Vert T(u)\right\Vert ^{2}+\mathrm{o}(1).$$  A
combination of the previous equalities completes the argument.

\vspace{1.5cm}
?
Compact perturbations of continuous linear operators have important connections in functional analysis (see, e.g.,  \cite{con} and  the famous \textquotedblleft
Scalar-plus-compact problem\textquotedblright\ of Lindenstrauss \cite{Lin}, solved by Argyros and Haydon \cite{Ar}).  In \cite{quae}, J. Kover showed the following
result:

\begin{proposition}
[J. Kover \cite{quae}]\label{hh} Let $H$ be a Hilbert space and $K, T:H\rightarrow H$
compact, resp. bounded linear operators. If
\[
\left\Vert T\right\Vert <\left\Vert T+K\right\Vert ,
\]
then $T+K$ is norm attaining.
\end{proposition}

In other words, if a compact perturbation of $T\in\mathcal{L}(H,H)$ has norm strictly
greater than $T$, then this perturbation is norm attaining.  We generalize this result by proving that it holds for all pairs $(E,F)$
satisfying the weak maximizing property. This result highlights the importance of the
weak maximizing property and motivates the investigation of the problem stated in the
Introduction. Proposition \ref{second-result} helps to understand that Proposition \ref{sepG} generalizes Proposition \ref{hh}.

\begin{proposition}
\label{sepG} Suppose that the pair $\left(  E,F\right)  $ has the weak maximizing
property and that $K, T:E\rightarrow F$ are compact, resp. bounded linear operators.
If
\begin{equation}
\left\Vert T\right\Vert <\left\Vert T+K\right\Vert , \label{wqqa}%
\end{equation}
then $T+K$ is norm attaining.\\

In particular, this applies when $E=\ell_{p}(\Gamma_{1})$ and
$F=\ell_{q}(\Gamma_{2})$, with $1<p<\infty$ and $1\leq q<\infty$ and $\Gamma_{1}$,
$\Gamma_{2}$ arbitrary sets of indexes.
\end{proposition}

Before giving the short argument, we remark that we are not really interested in
$T$ or in $K$ but rather in their sum  $T + K.$ The reason for this is that
it may be possible that  $||T|| \geq ||T + K||$ but, for some
other bounded, resp. compact operator  $T',$ resp. $K',$  we have
$ T + K  = T' + K'$   and also that $||T'|| < ||T'+K'||.$  In such a case, if $(E,F)$ has
the weak maximizing property then  $T + K = T' + K'$   is norm-attaining.
Also, assuming that $(E,F)$ has the weak maximizing property, if $T \in \mathcal L(E,F)$
and $\epsilon > 0$ are arbitrary, it is easy to show that there is a rank one operator $K:E \to F$ 
such that $\|K\| < \epsilon$ and that $\|T\| < \|T + K\|.$ Thus any $T$ can be approximated by a norm-attaining operator whose difference with $T$ has rank one. Compare with 
the result of Lindenstrauss in  \cite[Theorem 1]{Lin-isr}, but of course the result itself was subsumed
by Theorem 15 in Stegall's paper \cite{stegall}.\\

\begin{proof}
Suppose that $T+K$ fails to be norm attaining. Since every maximizing sequence $\left(
x^{n}\right)  _{n=1}^{\infty}$ for $T+K$ is weakly null, by the compactness of $K$ we
have that $(K\left(  x^{n}\right))_{n=1}^\infty$ converges to zero. Since $\left( x^{n}\right)
_{n=1}^{\infty}$ is maximizing for $T+K$ we know that
\[
\underset{n\longrightarrow\infty}{\lim}\left\Vert T\left(  x^{n}\right)  +K\left(
x^{n}\right)  \right\Vert =\left\Vert T+K\right\Vert ,
\]
and we get
\[
\underset{n\longrightarrow\infty}{\lim}\left\Vert T\left(  x^{n}\right) +K\left(
x^{n}\right)  \right\Vert \leq\sup_{n}\left\Vert T(x^{n})\right\Vert \leq\left\Vert
T\right\Vert <\left\Vert T+K\right\Vert ,
\]
which is a contradiction.
\end{proof}

If $E$ is non-reflexive, let $\varphi \in E^\ast$ be a functional that does not attain
its norm. If $F\neq\{0\}$, choose any non-zero $y^0 \in F.$ Then the compact operator $K:E \to F, \
K(x) = \varphi(x)y^0$ is not norm attaining. Taking $T = 0,$ we see that Proposition \ref{sepG}
fails.  This argument extends to prove the following result.

\begin{corollary}
\label{uuio}A pair $\left(  E,F\right)  $ has the weak maximizing property
for some $F\neq\{0\}$ if and only if $E$ is reflexive. In particular if $E$
has the weak maximizing property then $E$ is reflexive.
\end{corollary}

It is also worth mentioning that the inequality in (\ref{wqqa}) needs to be
strict. For example, consider%
\[
T:\ell_{2}\rightarrow\ell_{2},\text{ }T(x)=\left(  x_{1},\frac{x_{2}}%
{2},...,\frac{(n-1)x_{n}}{n},...\right)
\]
and%
\[
K:\ell_{2}\rightarrow\ell_{2},\text{ }K(x)=\left(  -x_{1},0,0,...\right)  .
\]
It is easy to check that $T+K$ is not norm attaining and $\left\Vert T+K\right\Vert
=\left\Vert T\right\Vert =1.$

\section{Norm attaining polynomials}

The investigation of norm attaining polynomials and multilinear mappings is a subject
which has attracted the attention of several authors (see \cite{acosta, ar, AFW, aron} and
references therein).

In this section $m\geq2$ will always denote a positive integer. If $1<p,q<\infty$ and
$mq<p$, then every continuous $m$-homogeneous polynomial $P:\ell_{p}\rightarrow\ell_{q}$
is sequentially weak-to-norm continuous (see, e.g.,  \cite[Exercise 2.65]{Din} for the more
general case of multilinear mappings). Hence, since $\ell_{p}$ is reflexive, one can
prove that $P$ attains its norm. So, the analogous version of Theorem \ref{ptp} (when
$mq<p$) works for $m-$homogeneous polynomials. Hence, the relevant question is whether
Theorem \ref{ptp} is true for the case $mq\geq p$. As we will show, the answer is
negative in general.

\subsection{The case $m\geq3$}

We remark that a counterexample for the case $F=\mathbb{R}$ also works for a general $F$
over $\mathbb{R}$. So, we will restrict ourselves to the case $F=\mathbb{R}$.

If $m\geq3,$ the polynomial $P_{m}:\ell_{2}\rightarrow\mathbb{R}$ defined by%
\[
P_{m}(x)=x_{1}^{m}+mx_{1}^{m-2}%
{\displaystyle\sum\limits_{j=2}^{\infty}}
\frac{j}{j+1}x_{j}^{2}%
\]
does not attain its norm of $2\big(\frac{m-2}{m-1}\big)^{\frac{m-2}{2}}\ \ (> 1 = P_m(e_1)),$ and there is a non-weakly null maximizing sequence for $P_{m}.$
In fact, one can see that $P_{m}$ does not attain its norm
and the sequence%
\[
\left(  \sqrt{  \frac{m-2}{m-1}  }e_{1}+\sqrt{1-\frac{m-2}{m-1}  }e_{n}\right)  _{n=2}^{\infty},%
\]
where $(e_{n})_{n=1}^{\infty}$ is the canonical basis of $\ell_{2}$, is maximizing and
non-weakly null.

\subsection{The case $m=2$}

Here we still have a counterexample when $F$ is infinite dimensional. In fact, consider
$P:\ell_{2}\rightarrow\ell_{2}$ defined by
\[
P(x)=x_{1}\left(  \frac{x_{2}}{2},\frac{2x_{3}}{3},\frac{3x_{4}}%
{4},...\right)  .
\]
Note that $\left\Vert P\right\Vert =1/2$ and $P$ does not attain its
norm. On the other hand, the sequence%
\[
\left(  \frac{1}{\sqrt{2}}e_{1}+\frac{1}{\sqrt{2}}e_{n}\right)
_{n=2}^{\infty}%
\]
is maximizing and
non-weakly null.

\subsection{The case $m=2,$ $E=\ell_{2}$ and $F=\mathbb{K}$}

Now we will show that, contrary to the other two cases, the case of
scalar-valued $2$-homogeneous polynomials has a positive result:

\begin{theorem}
\label{pppf}Let $P\colon\ell_{2}\rightarrow\mathbb{K}$ be a continuous $2$-homogeneous
polynomial. Then $P$ attains its norm if, and only if, there exists a non-weakly null
maximizing sequence for $P$. In particular, the pair $\left( \ell_{2},\mathbb{K}\right) $
has the $2-$homogeneous polynomial weak maximizing property.
\end{theorem}

\begin{proof}
As usual, just one implication needs a proof. Let $(x^{n})_{n=1}^{\infty}$ be
a non-weakly null maximizing sequence for $P.$ Let $u_{P}:\ell_{2}%
\rightarrow\ell_{2}^{\ast}$ be the linear operator defined by
\[
u_{P}(x)(y)=T_{P}(x,y).
\]
It is well-known that%
\[
\left\Vert u_{P}\right\Vert =\left\Vert T_{P}\right\Vert
\]
and since we are dealing with Hilbert spaces, an old result due to Banach  (see \cite{Bana} or  \cite[Proposition 1.44]{Din} for an easy available reference) asserts that the norm of any homogeneous polynomial coincides with the norm of its associated symmetric multilinear operator (see also  \cite[Theorem 2.1]{Sa} for a modern constructive approach), i.e.,
\[
\left\Vert P\right\Vert =\left\Vert T_{P}\right\Vert .
\]
So we have%
\[
\left\Vert P\right\Vert =\left\Vert T_{P}\right\Vert =\left\Vert u_{P}\right\Vert
\geq\left\Vert u_{P}(x^{n})\right\Vert \geq\left\Vert T_{P}(x^{n},x^{n})\right\Vert
=\left\Vert P(x^{n})\right\Vert
\]
and since $\left\Vert P(x^{n})\right\Vert$ converges to $\left\Vert P\right\Vert$ as $n$
goes to infinity, we conclude that
\[
\lim_{n\rightarrow\infty}\left\Vert u_{P}(x^{n})\right\Vert =\left\Vert
u_{P}\right\Vert .
\]
So, by Theorem \ref{ptp} we have that $u_{P}:\ell_{2}\rightarrow\ell _{2}^{\ast}$ attains
its norm at some norm-one vector $x^{0}\in\ell_{2}$.
Hence%
\[
\left\Vert u_{P}(x^{0})\right\Vert =\left\Vert u_{P}\right\Vert =\left\Vert
T_{P}\right\Vert .
\]
But, since $u_{P}(x^{0}):\ell_{2}\rightarrow\mathbb{K}$ is a continuous linear
functional, $u_{P}(x^{0})$ also attains its norm at some norm-one vector
$y^{0}\in\ell_{2}$. Therefore
\[
|T_{P}(x^{0},y^{0})| =|u_{P}(x^{0}%
)(y^{0})| =\left\Vert u_{P}(x^{0})\right\Vert =\left\Vert T_{P}\right\Vert
\]
and so $T_{P}:\ell_{2}\times\ell_{2}\rightarrow\mathbb{K}$ attains its norm at
$(x^{0},y^{0})\in\ell_{2}\times\ell_{2}$. Finally, thanks again to Banach's result (see \cite{Bana}) we conclude that $P$ attains its norm.
\end{proof}

The results in this section show that things work nicely, from the point of view of non-weakly null
maximizing sequences and norm attainment, for scalar-valued $2$-homogeneous polynomials. However
this is no longer the case 
for vector-valued $2$-homogeneous polynomials and for scalar-valued $m$-homogeneous polynomials with $m\geq 3$. This illustrates a pattern in the attempts to generalize linear results
to the polynomial setting. For example, in \cite[Section 4]{BMP} the validity of expected results resists one degree of homogeneity more: things work nicely for $2$-homogeneous polynomials and for scalar-valued $3$-homogeneous polynomials but deteriorate for vector-valued $3$-homogeneous polynomials and for $m$-homogeneous polynomials with $m\geq  4$. Further information is found, for example, in 
\cite{aron}.

\subsection{Completely continuous perturbations of homogeneous polynomials}

An $m$-homogeneous polynomial is said to be {\em completely continuous} when it is weak-to-norm continuous. For example, finite-type polynomials are always completely continuous.

>From the proof of Proposition \ref{sepG} one can check that a polynomial version is
possible if we replace the compact perturbation by a perturbation by a completely
continuous polynomial. More precisely, if the pair $(E,F)$ has the $m$-homogeneous
polynomial weak maximizing property, the following holds: For any continuous
$m$-homogeneous polynomial $P:E\rightarrow F$ and any  completely continuous
$m$-homogeneous polynomial $K:E\rightarrow F$ such that
\[
\left\Vert P\right\Vert <\left\Vert P+K\right\Vert ,
\]
the polynomial $P+K$ is norm attaining.

Since the pair $(\ell_{2},\mathbb{K})$ has the $2$-homogeneous polynomial weak maximizing
property (see Theorem \ref{pppf}), we see that  completely continuous perturbations of
$2$-homogeneous polynomials increasing their norms provide norm attaining polynomials. That is,
if $P:\ell_{2}\rightarrow\mathbb{K}$ is a continuous $2$-homogeneous polynomial,
$K:\ell_{2}\rightarrow\mathbb{K}$ is a completely continuous $2$-homogeneous polynomial and%
\[
\left\Vert P\right\Vert <\left\Vert P+K\right\Vert ,
\]
then $P+K$ is norm attaining.

\bigskip

\bigskip

\bigskip

We conclude with the following questions:

\begin{questions}
\begin{enumerate}
\item Are there  infinite dimensional Banach spaces $E, F_1$ and  $F_2$ for
	which the pair $(E,F_1)$ has the weak maximizing property but the pair $(E,F_2)$
fails to have it? In particular, if $(E,F)$ has the weak maximizing property, does
$(E,F')$ also have it for an isomorphic copy of $F?$ More generally, if the pair $(E,F)$ has
the weak maximizing property and if $E', F'$ are
isomorphic to $E, F,$ respectively, does the pair $(E',F')$ have the same property?\\
\item What is the relation between $E$ having the weak maximizing property and $E$ being reflexive? We conjecture that the two properties are not equivalent.
\end{enumerate}
\end{questions}

\vspace{.5cm}

\textbf{Acknowledgment.} We thank Geraldo Botelho for helpful conversations
and Mar\'ia Acosta and Lawrence Harris for providing us the references \cite{Bana} and
\cite{Sa}.

\end{document}